\documentclass{amsart}
\usepackage[russian]{babel}
\numberwithin{equation}{section}

\newtheorem{thm}{Theorem}[section]
\newtheorem{lem}[thm]{Lemma}
\newtheorem{cor}[thm]{Corollary}

\newtheorem{defin}[thm]{Definition}

\begin{document}
\title{Time-dependent source identification problem for a fractional Schr\"odinger
equation with the Riemann-Liouville derivative}

\author{Ravshan Ashurov}
\author{Marjona Shakarova}
\address{Institute of Mathematics, Uzbekistan Academy of Science,
Tashkent, Student Town str. 100174} \email{ashurovr@gmail.com}
\curraddr{National University of Uzbekistan named after Mirzo Ulugbek,  Tashkent, Student Town str. 100174} \email{shakarova2104@gmail.ru}

\small

\title[Time-dependent source identification problem ...] { Time-dependent source identification problem for a fractional Schrödinger
equation with the Riemann-Liouville derivative }

\begin{abstract} 
 The Schr\"odinger equation
$i \partial_t^\rho u(x,t)-u_{xx}(x,t) = p(t)q(x) + f(x,t)$ ( $0<t\leq T, \, 0<\rho<1$), with the Riemann-Liouville derivative is considered.   An inverse problem is investigated  in which, along with $u(x,t)$, also a time-dependent  factor  $p(t)$  of the source function is unknown. To solve this inverse problem, we take the additional condition $ B [u (\cdot,t)] = \psi (t) $ with an arbitrary bounded linear functional $ B $. Existence and uniqueness theorem for
the solution to the problem under consideration is proved. Inequalities of stability are obtained. The applied method allows us to study a similar problem by taking instead of $d^2/dx^2$ an arbitrary elliptic differential operator $A(x, D)$, having a compact inverse.

\end{abstract}

\maketitle

\section{Introduction}

The fractional integration of order $ \sigma <0 $ of function $ h (t) $ defined on $ [0, \infty) $ has the form (see e.g. \cite{Pskhu}, p.14, \cite{SU}, Chapter 3)
\[
J_t^\sigma h(t)=\frac{1}{\Gamma
(-\sigma)}\int\limits_0^t\frac{h(\xi)}{(t-\xi)^{\sigma+1}} d\xi,
\quad t>0,
\]
provided the right-hand side exists. Here $\Gamma(\sigma)$ is
Euler's gamma function. Using this definition one can define the
 Riemann-Liouville fractional derivative of order $\rho$:

$$
\partial_t^\rho h(t)= \frac{d}{dt}J_t^{\rho-1} h(t).
$$

Note that if $\rho=1$, then the fractional derivative coincides with
the ordinary classical derivative of the first order: $\partial_t h(t)= (d/dt) h(t)$.

Let $\rho\in(0,1) $ be a fixed number and $\Omega =(0,\pi) \times (0, T]$.
Consider the following initial-boundary value problem for the Shr\"odinger equation
\begin{equation}\label{prob1}
\left\{
\begin{aligned}
&i\partial_t^\rho u(x,t) - u_{xx}(x,t) =p(t)q(x)+ f(x,t),\quad (x,t)\in \Omega;\\
&u(0,t)=u(\pi, t)=0, \quad 0\leq t\leq T;\\
&\lim\limits_{t\rightarrow 0}J_t^{\rho-1}u(x,t)=\varphi(x), \quad 0\leq x\leq \pi,
\end{aligned}
\right.
\end{equation}
where $t^{1-\rho}p(t)$, $t^{1-\rho}f(x,t)$ and $\varphi(x), q(x)$ are continuous functions in the closed domain $\overline{\Omega}$. This problem is also called the \textit{forward problem}.

If $p(t)$ is a known function, then under certain conditions on the given functions a solution to problem
(\ref{prob1}) exists and it is unique (see, e.g., \cite{AOLob}).

We note the following property of the Riemann-Liouville integrals, which simplifies the verification of the initial condition in problem (\ref{prob1}) (see e.g. \cite{Pskhu}, p.104):

\begin{equation} \label{property}
	\lim\limits_{t\rightarrow +0}J_t^{\alpha-1} h(t) = \Gamma
	(\alpha)\lim\limits_{t\rightarrow +0} t^{1-\alpha} h(t).
\end{equation}
From here, in particular, it follows that the solution of the forward problem can have a
singularity at zero $t = 0$ of order $t^{\rho-1}$.

Let $C[0, l]$ be the set of continuous functions defined on $[0,l]$ with the standard max-norm $||\cdot||_{C[0,l]}$. The purpose
of this paper is not only to find a solution $u(x,t)$, but also to determine
the time-dependent part $p(t)$ of the source function. To solve this time-dependent source identification problem
 one needs an extra
condition. Following the papers of A. Ashyralyev et al. \cite{Ashyr1}-\cite{Ashyr3} we consider the additional condition in a rather general form:
\begin{equation}\label{ad}
B[u(\cdot,t)]=\psi(t), \quad 0\leq t \leq T,
\end{equation}
where $B: C[0, \pi]\rightarrow R$ is a given bounded linear functional: $||B[h(\cdot, t)]||_{C[0,T]}\leq b||h(x,t)||_{C(\overline{\Omega})}$, and $\psi(t)$ is a given continuous function. For example, as the functional $B$ one can take  $B[u(\cdot,t)]=u(x_0, t)$, $x_0\in [0, \pi]$, or $B[u(\cdot,t)]=\int_{0}^{\pi} u(x,t) dx$, or a linear combination of these two functionals. 

We call the initial-boundary value problem (\ref{prob1}) together with additional condition (\ref{ad})
 \emph{the inverse problem}.

 When solving the inverse problem, we will investigate the Cauchy and initial-boundary value problems for various differential equations. In this case, by the solution of the problem we mean the classical solution, i.e. we will assume that all derivatives and functions involved in the equation are continuous with respect to the variable $x$ and $t$ in a open set. As an example, let us give the definition of the solution to the inverse problem.

\begin{defin}\label{def} A pair of functions $\{u(x,t), p(t)\}$ with the properties
\begin{enumerate}
	\item
	$\partial_t^\rho u(x,t), u_{xx}(x,t)\in C(\Omega)$,
	\item$t^{1-\rho}u(x,t)\in C(\overline{\Omega})$,
	\item
$t^{1-\rho}p(t)\in C[0,T]$,
\end{enumerate}
and satisfying conditions
(\ref{prob1}), (\ref{ad})  is called \textbf{the
 solution} of the inverse problem.
\end{defin}

Note that condition 3 in this definition is taken in order to cover a wider class of functions, as function $p(t)$. In this regard, it should be noted that, to the best of our knowledge,  the time-dependent source identification problem for equations with the Riemann-Liouville derivative is being studied for the first time.

Taking into account the boundary conditions in problem (\ref{prob1}), it is convenient for us to introduce the H\"older classes as follows. Let $\omega_g(\delta)$ be the modulus of continuity of function $g(x)\in C[0, \pi]$, i.e.
\[
\omega_g(\delta)=\sup\limits_{|x_1-x_2|\leq\delta}
|g(x_1)-g(x_2)|, \quad x_1, x_2\in [0,\pi].
\]
If $\omega_g(\delta)\leq C \delta^a$ is true for some $a>0$, where $C$ does not depend on $\delta$ and $g(0)=g(\pi)=0$, then $g(x)$ is said to belong to the H\"older class $C^a[0,\pi]$. Let us denote the smallest of all such constants $C$ by $||g||_{C^a[0,\pi]}$. Similarly, if the continuous function $h(x, t)$ is defined on $[0,\pi]\times[0, T]$, then the value
\[
\omega_h(\delta ;t)=\sup\limits_{|x_1-x_2|\leq\delta} |h(x_1,
t)-h(x_2, t)|, \quad x_1, x_2\in [0,\pi]
\]
is the modulus of continuity of function $h(x, t)$ with respect to the variable
$x$. In case when $\omega_h(\delta; t)\leq C \delta^a$, where
$C$ does not depend on $t$ and $\delta$ and $h(0,t)=h(\pi,t)=0,\,\, t\in [0,T]$, then we say that
$h(x, t)$ belongs to the H\"older class $C_x^a(\overline{\Omega})$. Similarly, we denote the smallest constant $C$ by  $||h||_{C_x^a(\overline{\Omega})}$.

Let $C_{2,x}^a(\overline{\Omega})$ denote the class of functions $h(x,t)$ such that $h_{xx}(x,t)\in C_x^a(\overline{\Omega})$  and  $h(0,t)=h(\pi,t)=0,\, t\in [0,T]$. Note that condition $h_{xx}(x,t)\in C_x^a(\overline{\Omega})$  implies that $h_{xx}(0,t)=h_{xx}(\pi,t)=0,\,\, t\in [0,T]$. For a function of one variable $g(x)$, we introduce classes $C_{2}^a[0, \pi]$ in a similar way.

\begin{thm}\label{main} Let $a>\frac{1}{2}$ and the following conditions be satisfied
	
	\begin{enumerate}
		\item
		$t^{1-\rho}f(x,t)\in C^a_x(\overline{\Omega})$,
		\item$\varphi \in C^a[0, \pi]$,
		\item
		$t^{1-\rho}\psi(t),\,\, t^{1-\rho}\partial_t^\rho\psi(t) \in C[0,T]$,
		\item$q \in C_2^a[0, \pi],\,\, B[q(x)]\neq 0$.
		
	\end{enumerate}
	Then the inverse problem has a unique solution $\{u(x,t), p(t)\}$.

\end{thm}

Everywhere below we denote by $a$ an arbitrary number greater than $1/2$: $a>1/2$.

If we additionally require that the initial function $\varphi\in C_2^a[0, \pi]$, then we can establish the following result on the stability of the solution of the inverse problem.
\begin{thm}\label{estimate} Let assumptions of Theorem \ref{main} be satisfied and let $\varphi\in C_2^a[0, \pi]$. Then  the solution to the invest problem obeys the stability estimate
\[
    ||t^{1-\rho}\partial_t^\rho u||_{C(\overline{\Omega})} + ||t^{1-\rho} u_{xx}||_{C(\overline{\Omega})} + ||t^{1-\rho}p||_{C[0,T]}
\]    
    \begin{equation}\label{se}
    	\leq C_{\rho, q, B} \bigg[  ||\varphi_{xx}||_{C^a[0, \pi]}  + ||t^{1-\rho}\psi||_{C[0,T]}+ ||t^{1-\rho}\partial_t^\rho\psi||_{C[0,T]} + ||t^{1-\rho}f(x,t)||_{C_x^a(\overline{\Omega})}\bigg],
\end{equation}
where $C_{\rho, q, B}$ is a constant, depending only on $\rho, q$ and $B$.
\end{thm}

It should be noted that the method proposed here, based on the Fourier method, is applicable to the equation in (\ref{prob1}) with an arbitrary elliptic differential operator $A(x, D)$ instead of $d^2/dx^2$, if only the corresponding spectral problem has a complete system of orthonormal eigenfunctions in $ L_2(G), \, G\subset R^N$.

The interest in the study of source (right-hand side of the equation $F(x,t)$) identification inverse problems is caused primarily in connection with practical requirements in various branches of mechanics, seismology, medical tomography, and geophysics (see e.g. the survey paper Y. Liu et al. \cite{Hand1}). The identification of $F(x,t)=h(t)$ is appropriate, for example, in cases of accidents at nuclear power plants, when it can be assumed that the location of the source is known, but the decay of the radiation power over time is unknown and it is important to estimate it. On the other hand, one example of the identification of $F(x,t)=g(x)$ can be the detection of illegal wastewater discharges, which is a serious problem in some countries.

The inverse problem of determining the source function $F$ with the final time observation have been well studied and many theoretical researches have been published for classical partial differential equations (see, e.g.  monographs Kabanikhin \cite{Kab1} and Prilepko, Orlovsky, and Vasin \cite{prilepko}). As for fractional differential equations, it is possible to construct theories parallel to \cite{Kab1}, \cite{prilepko}, and the work is now ongoing. Let us mention only some of these works (a detailed review can be found in \cite{Hand1}).

It should be noted right away that for the abstract case of the source function $F(x, t)$ there is currently no general closed theory. Known results deal with separated source
term $F(x, t) = h(t)g(x)$. The appropriate choice of the overdetermination
depends on the choice whether the unknown is $h(t)$ or $g(x)$.

Relatively fewer works are devoted to the case when the unknown is the function $h(t)$ (see the survey work \cite{Hand1} and \cite{Yama11} for the case of subdiffusion equations, and, for example, \cite{Ashyr1}-\cite{Ashyr3} for the classical heat equation).

Uniqueness questions in the inverse problem of finding a function $g(x)$ in fractional diffusion equations with the sourse function $g(x)h(t)$ has been studied in, e.g. \cite{Niu}, \cite{MS}, \cite{MS1}.

In many papers, authors have considered an equation, in which $ h (t) \equiv 1 $ and $ g (x) $ is unknown (see, e.g. \cite{12} - \cite{AF1}).
The case of subdiffusion equations whose elliptic part is an ordinary differential expression is considered in \cite{12} - \cite{Tor}. The authors of the articles \cite{20} - \cite{24} studied subdiffusion equations in which the elliptic part is either a Laplace operator or a second-order selfadjoint operator. The paper \cite{25}
studied the inverse problem for the abstract subdiffution equation. In this article \cite{25} and
most other articles, including \cite{20} - \cite{23}, the Caputo derivative is
used as a fractional derivative. The subdiffusion equation considered in the recent papers \cite{AOLob}, \cite{AODif}   contains the fractional Riemann-Liouville derivative, and the elliptical part is an arbitrary elliptic expression of order $m$.
In \cite{13} and \cite{24}, the fractional
derivative in the subdiffusion equation is a two-parameter
generalized Hilfer fractional derivative.
Note also that the papers \cite{13}, \cite{20}, \cite{23} contain a survey of papers dealing with inverse problems of determining the right-hand side
of the subdiffusion equation.

In \cite{24}, \cite{KirSal}, \cite{Ali}, non-self-adjoint differential operators (with nonlocal boundary conditions) were taken as elliptical part of the equation, and solutions to the inverse problem were found in the form of biortagonal series.

In \cite{AF1}, the authors considered an inverse problem for simultaneously determining the order of the Riemann-Liouville fractional derivative and the source function in the subdiffusion equations. Using the classical Fourier method, the authors proved the uniqueness and existence of a solution to this inverse problem.

It should be noted that in all of the listed works, the Cauchy
conditions in time are considered (an exception is work \cite{Saima}, where the integral condition is set with respect to the variable $t$). In the recent paper \cite{AFFF}, for the best of our knowledge, an inverse problem for subdiffusion equation
with a nonlocal condition in time is considered for the first
time.

The papers \cite{AF2} - \cite{AF3} deal with the inverse problem of determining the order of the fractional derivative in the subdiffusion equation and in the wave equation, respectively.

Time-dependent source identification problem (\ref{prob1}) for classical Schr\"odinger
type equations (i.e. $\rho=1$) with additional condition (\ref{ad}) was for the first time investigated in papers of A. Ashyralyev at al. \cite{Ashyr1}-\cite{Ashyr3}. To investigate the inverse  problem (\ref{prob1}), (\ref{ad}) we borrow some
original ideas from these papers.

\section{Preliminaries}

In this section, we  recall some information about Mittag-Leffler functions, differential and integral equations, which we will use in the following sections.

For $0 < \rho < 1$ and an arbitrary complex number $\mu$, by $
E_{\rho, \mu}(z)$ we denote the Mittag-Leffler function  of complex argument $z$ with two
parameters:
\begin{equation}\label{ml}
E_{\rho, \mu}(z)= \sum\limits_{k=0}^\infty \frac{z^k}{\Gamma(\rho
k+\mu)}.
\end{equation}
If the parameter $\mu =1$, then we have the classical
Mittag-Leffler function: $ E_{\rho}(z)= E_{\rho, 1}(z)$.

Since $E_{\rho, \mu}(z)$  is an analytic function of $z$, then it is bounded for $|z|\leq 1$. On the other hand the well known asymptotic estimate of the
Mittag-Leffler function has the form (see, e.g.,
\cite{Dzh66}, p. 133):

\begin{lem}\label{ml} Let $\mu$ be an arbitrary complex number. Further let $\alpha$ be a fixed number, such that $\frac{\pi}{2}\rho<\alpha<\pi \rho$, and $\alpha \leq |\arg z|\leq \pi$. Then the following asymptotic estimate holds
\[
E_{\rho, \mu}(z)= -
\sum\limits_{k=1}^2\frac{z^{-k}}{\Gamma(\rho-k\mu)} + O(|z|^{-3}),
\,\, |z|>1.
\]
\end{lem}

We can choose the parameter $\alpha$ so that the following estimate is valid:
\begin{cor}\label{ml1} For any $t\geq 0$ one has
\[
|E_{\rho, \mu}(it)|\leq \frac{C}{1+t},
\]
where constant $C$ does not depend on $t$ and $\mu$.
\end{cor}

We will also use a coarser estimate with positive number $\lambda$
and $0<\varepsilon<1$:
\begin{equation}\label{ml2}
|t^{\rho-1} E_{\rho,\rho}(-i\lambda t^\rho)|\leq
\frac{Ct^{\rho-1}}{1+\lambda t^\rho}\leq C
\lambda^{\varepsilon-1} t^{\varepsilon\rho-1}, \quad t>0,
\end{equation}
which is easy to verify. Indeed, let $t^\rho\lambda<1$, then $t<
\lambda^{-1/\rho}$ and
$$
t^{\rho -1} = t^{\rho-\varepsilon\rho} t^{\varepsilon\rho-1} <
\lambda^{\varepsilon-1}t^{\varepsilon\rho-1}.
$$
If $t^\rho\lambda\geq 1$, then $\lambda^{-1}\leq t^\rho$ and
$$
\lambda^{-1} t^{-1}=\lambda^{-1+\varepsilon}
\lambda^{-\varepsilon} t^{-1}\leq
\lambda^{\varepsilon-1}t^{\varepsilon\rho-1}.
$$
\begin{lem}\label{du} Let $t^{1-\rho}g(t)\in C[0, T]$. Then the unique solution of the Cauchy problem
\begin{equation}\label{prob.y}
\left\{
\begin{aligned}
&i\partial_t^\rho y(t) + \lambda y(t) =g(t),\quad 0<t\leq T;\\
&\lim\limits_{t\rightarrow 0}J_t^{\rho-1}y(t)=y_0
\end{aligned}
\right.
\end{equation}
has the form
\[
y(t)=t^{\rho-1} E_{\rho, \rho}(i\lambda t^\rho) y_0 - i \int\limits_0^t (t-s)^{\rho-1} E_{\rho,\rho} (i\lambda (t-s)^\rho) g(s) ds.
\]
\end{lem}
\begin{proof} Multiply equation (\ref{prob.y})  by $(-i)$ and then apply formula (7.2.16) of \cite{Gor}, p. 174 (see also \cite{AshCab}, \cite{AshCab2}).

\end{proof}

Let us denote by $A$ the operator $-d^2/dx^2$ with the domain $D(A) = \{v(x)\in W_2^2(0, \pi): v(0)=v(\pi)=0\}$, where $W_2^2(0, \pi)$ - the standard Sobolev space. Operator $A$ is selfadjoint in $L_2(0, \pi)$ and has the complete in $L_2(0, \pi)$ set of  eigenfunctions $\{v_k(x) = \sin kx\}$ and eigenvalues $\lambda_k=k^2$, $k=1,2,...$.

Consider the operator $E_{\rho, \mu} (it A)$, defined by the spectral theorem of J. von Neumann:
\[
E_{\rho, \mu} (it A)h(x,t) = \sum\limits_{k=1}^\infty E_{\rho,\mu} (it \lambda_k) h_k(t) v_k(x),
\]
here and everywhere below, by $h_k(t)$  we will denote the Fourier coefficients of a function $h(x,t)$: $h_k(t)=(h(x,t),v_k)$, $(\cdot, \cdot)$ stands for the scalar product in $L_2(0, \pi)$. This series converges in the $L_2(0, \pi)$ norm. But we need to investigate the uniform convergence of this series in $\Omega$. To do this, we recall the following statement.

\begin{lem}\label{Zyg}Let  $g\in C^a[0, \pi]$.
Then for any $\sigma\in [0, a-1/2)$ one has
\[
\sum\limits_{k=1}^\infty k^\sigma|g_k|<\infty.
\]	
\end{lem}

For $\sigma=0$ this assertion coincides with the well-known theorem of S. N. Bernshtein on the absolute convergence of trigonometric series and is proved in exactly the same way as this theorem. For the convenience of readers, we recall the main points of the proof (see, e.g. \cite{Zyg}, p. 384).

\begin{proof}

In theorem (3.1) of A. Zygmund \cite{Zyg}, p. 384, it is proved
that for an arbitrary function $g(x)\in C[0,\pi]$, with the properties
$g(0)=g(\pi)=0$, one has the estimate
\[
\sum\limits_{k=2^{n-1}+1}^{2^n} |g_k|^2 \leq
\omega_g^2\bigg(\frac{1}{2^{n+1}}\bigg).
\]
Therefore, if $\sigma\geq 0$, then by the Cauchy-Bunyakovsky inequality
\[
\sum\limits_{k=2^{n-1}+1}^{2^n} k^\sigma |g_k| \leq
\bigg(\sum\limits_{k=2^{n-1}+1}^{2^n}
|g_k|^2\bigg)^{\frac{1}{2}}\bigg(\sum\limits_{k=2^{n-1}+1}^{2^n}
k^{2\sigma}\bigg)^{\frac{1}{2}}\leq C
2^{n(\frac{1}{2}+\sigma)}\omega_g\bigg(\frac{1}{2^{n+1}}\bigg),
\]
and finally
\[
\sum\limits_{k=2}^{\infty} k^\sigma|g_k|=
\sum\limits_{n=1}^\infty\sum\limits_{k=2^{n-1}+1}^{2^n} k^\sigma|g_k| \leq  C \sum\limits_{n=1}^\infty
2^{n(\frac{1}{2}+\sigma)}\omega_g\bigg(\frac{1}{2^{n+1}}\bigg).
\]
Obviously, if $\omega_g(\delta)\leq C \delta^a$, $a>1/2$ and $0<\sigma<a - 1/2$, then
the last series converges:
\[
\sum\limits_{k=2}^{\infty} k^\sigma|g_k|\leq C ||g||_{C^a[0,\pi]}.
\]

\end{proof}

\begin{lem}\label{EA}Let 
$h(x,t)\in C_x^a(\overline{\Omega})$ 
	Then $E_{\rho, \mu} (it A)h(x,t)\in C(\overline{\Omega})$ and  $\frac{\partial^2}{\partial x^2} E_{\rho, \mu} (it A)h(x,t)\in C([0, \pi]\times(0,T])$. Moreover, the following estimates hold
    \begin{equation}\label{E}
    ||E_{\rho, \mu} (it A)h(x,t)||_{C(\overline{\Omega})} \leq C ||h||_{C_x^a(\overline{\Omega})};
    \end{equation}
\begin{equation}\label{AE}
   \bigg |\bigg|\frac{\partial^2}{\partial x^2} E_{\rho, \mu} (it A)h(x,t)\bigg|\bigg|_{C[0, \pi]} \leq C t^{-1} ||h||_{C_x^a(\overline{\Omega})}, \,\, t>0.
    \end{equation}

If 
$h(x,t)\in C_{2,x}^a(\overline{\Omega})$,
then

\begin{equation}\label{AEA}
   \bigg |\bigg|\frac{\partial^2}{\partial x^2} E_{\rho, \mu} (it A)h(x,t)\bigg|\bigg|_{C(\overline{\Omega})} \leq C ||h_{xx}||_{C_x^a(\overline{\Omega})}.
\end{equation}

\end{lem}
\begin{proof}By definition one has
    \[
    |E_{\rho, \mu} (it A)h(x,t)|=\bigg|\sum\limits_{k=1}^\infty E_{\rho, \mu} (it \lambda_k) h_k(t) v_k(x)\bigg|\leq\sum\limits_{k=1}^\infty |E_{\rho, \mu} (it \lambda_k) h_k(t)|.
    \]
    Corollary \ref{ml1} and Lemma \ref{Zyg} imply
    \[
    |E_{\rho, \mu} (it A)h(x,t)|\leq C \sum\limits_{k=1}^\infty \bigg|\frac{ h_k(t)}{1+t \lambda_k}\bigg|\leq C ||h||_{C_x^a(\overline{\Omega})}.
    \]
On the other hand,
    \[
   \bigg |\frac{\partial^2}{\partial x^2} E_{\rho, \mu} (it A)h(x,t)\bigg|\leq C \sum\limits_{k=1}^\infty \bigg|\frac{\lambda_k h_k(t)}{1+t \lambda_k}\bigg|\leq C t^{-1} ||h||_{C_x^a(\overline{\Omega})}, \,\, t>0.
    \]

    If $h(x,t)\in C_{2,x}^a(\overline{\Omega})$, then $h_k(t)=-\lambda_k^{-1} (h_{xx})_k(t)$. Therefore,
 \[
\bigg|\frac{\partial^2}{\partial x^2} E_{\rho, \mu} (it A)h(x,t)\bigg|\leq C ||h_{xx}||_{C_x^a(\overline{\Omega})}, \,\, 0\leq t\leq T.
\]
\end{proof}

\begin{lem}\label{EAintL}Let $t^{1-\rho}g(x,t)\in C_x^a(\overline{\Omega})$. Then  there is a positive constant $c_1$ such that 
	
	\begin{equation}\label{EAint}
		\bigg|t^{1-\rho}\int\limits_0^t (t-s)^{\rho-1}  E_{\rho, \rho} (i(t-s)^\rho A)g(x,s)ds \bigg| \leq c_1 \frac{t^\rho}{\rho}||t^{1-\rho}g||_{C_x^a(\overline{\Omega})}.
	\end{equation}

	\end{lem}

\begin{proof}

Apply estimate (\ref{E}) to get
    \[
    	\bigg|t^{1-\rho}\int\limits_0^t (t-s)^{\rho-1}  E_{\rho, \rho} (i(t-s)^\rho A)g(x,s)ds \bigg| \leq C t^{1-\rho}\int_{0}^{t}(t-s)^{\rho-1}s^{\rho-1} ds\cdot ||t^{1-\rho}g||_{C_x^a(\overline{\Omega})}.
    \]
For the integral one has
\begin{equation}\label{int}
	\int_{0}^{t}(t-s)^{\rho-1}s^{\rho-1} ds=	\int\limits_0^{\frac{t}{2}} \cdot + \int\limits_{\frac{t}{2}}^{t} \cdot\leq \frac{2^{2(1-\rho)}}{\rho} t^{2\rho-1}.
	\end{equation}
Denoting $c_1=4C$ we obtain the assertion of the lemma.
\end{proof}

\begin{cor}\label{g1g2} If  function $g(x,t)$ can be represented in the form $g_1(x) g_2(t)$, then the right-hand side of estimate (\ref{EAint}) has the form:
	\[
	c_1 \frac{t^\rho}{\rho}||g_1||_{C^a[0,\pi]}||t^{1-\rho}g_2||_{C[0,T
		]}.
	\]
	
	\end{cor}

\begin{lem}\label{ep} Let $t^{1-\rho}g(x,t)\in C_x^a(\overline{\Omega})$. Then 
    \[
    \bigg|\bigg|\int\limits_0^t(t-s)^{\rho-1} \frac{\partial^2}{\partial x^2} E_{\rho, \rho}(i(t-s)^\rho A)g(x,s) ds   \bigg|\bigg|_{C(\overline{\Omega})}\leq C ||t^{1-\rho}g||_{C_x^a(\overline{\Omega})}.
    \]
\end{lem}

\begin{proof}Let
    \[
        S_j(x,t)= \sum\limits_{k=1}^j
        \left[\int\limits_{0}^t(t-s)^{\rho-1} E_{\rho, \rho}(i\lambda_k(t-s)^\rho )g_k(s) ds\right] \lambda_k  v_k(x).
    \]

    Choose $\varepsilon$ so that $0<\varepsilon<a-1/2$ and apply the inequality (\ref{ml2}) to get 
    \[
    |S_j(t)|\leq C\sum\limits_{k=1}^j  \int\limits_0^t
    (t-s)^{\varepsilon\rho-1}s^{\rho-1}\lambda_k^{\varepsilon}|s^{1-\rho}g_k(s)|
    ds.
    \]
    By Lemma \ref{Zyg} we have
    \[
   |S_j(t)|\leq C ||t^{1-\rho}g||_{C_x^a(\overline{\Omega})}
    \]
    and since
     \[
    \int\limits_0^t(t-s)^{\rho-1} \frac{\partial^2}{\partial x^2} E_{\rho, \rho}(i(t-s)^\rho A)h(s) ds =\sum\limits_{j=1}^\infty S_j(t),
    \]
     the last inequality implies the assertion of the lemma.

    \end{proof}

\begin{lem}\label{prob.w}
Let $t^{1-\rho}G(x,t)\in C_x^a(\overline{\Omega})$ and  $\varphi\in C^a[0,\pi]$. Then the unique solution of the following initial-boundary value problem
\begin{equation}\label{prob2}
    \left\{
    \begin{aligned}
        &i \partial_t^\rho w(x,t) - w_{xx}(x,t) =G(x,t),\,\, 0 < t \leq T;\\
        &w(0,t)=w(\pi,t)=0, \,\,  0 < t \leq T;\\
        &\lim\limits_{t\rightarrow 0}J_t^{\rho-1} w(x,t) =\varphi(x),\,\, 0\leq x\leq \pi,
    \end{aligned}
    \right.
\end{equation}
has the form
\[
w(x,t)=t^{\rho-1} E_\rho(i t^\rho A) \varphi(x) - i \int\limits_0^t (t-s)^{\rho-1} E_{\rho, \rho} (i (t-s)^\rho A) G(x,s) ds.
\]
\end{lem}

\begin{proof}
According to the Fourier method, we will seek the solution to this problem in the form
\[
w(x,t) =\sum\limits_{k=1}^\infty T_k(t) v_k(x),
    \]
    where $T_k(t)$ are the unique solutions of the problems
    \begin{equation}\label{prob.T}
        \left\{
        \begin{aligned}
            &i \partial_t^\rho T_k + \lambda_k T_k(t) =G_k(t),\quad 0 < t \leq T;\\
            &\lim\limits_{t\rightarrow 0}J_t^{\rho-1} T_k(t) =\varphi_k,
        \end{aligned}
        \right.
    \end{equation}

Lemma \ref{du} implies
\[
T_k(t)=t^{\rho-1} E_\rho(i\lambda_k t^\rho) \varphi_k - i \int\limits_0^t (t-s)^{\rho-1} E_{\rho, \rho} (i\lambda_k (t-s)^\rho) G_k(s) ds.
\]
Hence the solution to problem (\ref{prob2}) has the form
\[
w(x,t)=t^{\rho-1} E_\rho(i t^\rho A) \varphi(x) - i \int\limits_0^t (t-s)^{\rho-1} E_{\rho, \rho} (i (t-s)^\rho A) G(x,s) ds.
\]
Note that the existence of the first term follows from estimate (\ref{E}), and the existence of the second term follows from Lemma \ref{EAintL}.

By Lemma \ref{ep} and estimate (\ref{AE}), we obtain: $w_{xx}(x,t)\in C(\Omega)$. Since $i\partial_t^\rho w(x,t)= -w_{xx}(x,t) + G(x,t)$, then $\partial_t^\rho w(x,t)\in C(\Omega)$.

The uniqueness of the solution can be proved by the standard
technique based on completeness of the set of eigenfunctions
$\{v_k(x)\}$ in $L_2(0,\pi)$ (see, e.g., \cite{AOLob}).

    \end{proof}

Let $t^{1-\rho}F(x,t)\in C(\overline{\Omega})$ and $g(x)\in C^a[0,\pi]$. Consider the Volterra integral equation
\begin{equation}\label{VE}
w(x,t)=F(x,t)+\int\limits_0^t  (t-s)^{\rho-1} E_{\rho, \rho} (i (t-s)^\rho A) g(x
) B[w(\cdot,s)]ds.
\end{equation}

\begin{lem}\label{VElem}
There is a unique solution $t^{1-\rho}w\in C(\overline{\Omega})$  to the integral equation (\ref{VE}).
\end{lem}
\begin{proof} Equation (\ref{VE}) is similar to the equations considered in the book \cite{Kil}, p. 199, Eq. (3.5.4) and it is solved in essentially the same way. Let us remind the main points.

Equation (\ref{VE}) makes sense in any interval $[0, t_1]\in [0,T]$, $(0<t_1<T)$. Choose $t_1$ such that
\begin{equation}\label{t1}
c_1 b ||g||_{C^a[0,\pi]} \frac{t_1^\rho}{\rho}<1
\end{equation}
and prove the existence of a unique solution $t^{1-\rho}w(x, t)\in C([0,\pi]\times[0, t_1])$ to the equation (\ref{VE}) on the interval $[0, t_1]$ (here the constant $c_1$ is taken from estimate (\ref{EAint}), see Corollary \ref{g1g2}). For this we use the Banach fixed point theorem for the space $C([0,\pi]\times[0, t_1])$ with the weight function $t^{1-\rho}$  (see, e.g., \cite{Kil}, theorem 1.9, p. 68), where the distance is given by
\[
d(w_1, w_2) = ||t^{1-\rho}[w_1(x,t)- w_2(x,t)]||_{C([0,\pi]\times [0, t_1])}.
\]
Let us denote the right-hand  side of equation (\ref{VE}) by $Pw(x,t)$, where $P$ is the corresponding linear operator. To apply the Banach fixed point theorem we have to prove the following:

(a) if $t^{1-\rho}w(x,t)\in C([0,\pi]\times[0, t_1])$, then $t^{1-\rho}Pw(x,t)\in C([0,\pi]\times[0, t_1])$;

(b) for any $t^{1-\rho}w_1, t^{1-\rho}w_2 \in C([0,\pi]\times[0, t_1])$ one has
\[
d(Pw_1,Pw_2)\leq \delta\cdot d(w_1,w_2), \,\, \delta<1.
\]

Lemmas \ref{EA} and \ref{EAintL} imply condition (a). On the other hand, thanks to (\ref{EAint}) (see Corollary \ref{g1g2}) we arrive at
\[
\bigg|\bigg|t^{1-\rho}\int\limits_0^t  (t-s)^{\rho-1} E_{\rho, \rho} (i (t-s)^\rho A) g(x) B[w_1(\cdot,s)-w_2(\cdot,s)]ds\bigg|\bigg|_{C([0,\pi]\times [0, t_1])}\leq\delta\cdot d(w_1,w_2),
\]
where $\delta =c_1 b ||g||_{C^a[0,\pi]} \frac{t_1^\rho}{\rho}<1$ since condition (\ref{t1}).

Hence by the Banach fixed point theorem, there exists a unique solution $t^{1-\rho} w^\star (x,t)\in C([0,\pi]\times [0, t_1])$ to equation (\ref{VE}) on the interval $[0, t_1]$ and this solution is a limit of the convergent
sequence $w_n(x,t)=P^n F(x,t)= P P^{n-1} F(x,t)$:
\[
\lim\limits_{n\rightarrow \infty} d(w_n(x,t),w^\star(x,t))=0.
\]

Next we consider the interval $[t_1, t_2]$, where $t_2=t_1+l_1<T$, and $l_1>0$. Rewrite the equation (\ref{VE}) in the form

\[
w(x,t)=F_1(x,t)+\int\limits_{t_1}^t  (t-s)^{\rho-1} E_{\rho, \rho} (i (t-s)^\rho A) g(x) B[w(\cdot,s)]ds,
\]
where
\[
F_1(x,t)=F(x,t)+ \int\limits_{0}^{t_1}  (t-s)^{\rho-1} E_{\rho, \rho} (i (t-s)^\rho A) g(x) B[w(\cdot,s)]ds,
\]
is a known function, since the function $w(x,t)$ is uniquely defined on the interval $[0, t_1]$. Using the same arguments
as above, we derive that there exists a unique solution $t^{1-\rho}w^\star(x,t)\in C([0,\pi]\times[t_1, t_2])$ to equation (\ref{VE}) on the interval $[t_1, t_2].$
Taking the next interval $[t_2, t_3]$, where $t_3=t_2+l_2<T$, and $l_2>0$, and repeating this process (obviously, $l_n>l_0>0$), we conclude that there exists a unique solution $t^{1-\rho}w^\star(x,t)\in C([0,\pi]\times[0, T])$ to equation (\ref{VE}) on the interval $[0, T]$, and this solution is a limit of the convergent
sequence $t^{1-\rho}w_n(x,t)\in C([0,\pi]\times[0, T])$:
\begin{equation}\label{wn}
\lim\limits_{n\rightarrow \infty} ||t^{1-\rho} [w_n(x,t)-w^\star(x,t)]||_{C(\overline{\Omega})}=0,
\end{equation}
with the choice of certain $w_n$ on each $[0, t_1], \cdots [t_{L-1}, T]$.
\end{proof}

We need the following kind of Gronwall's inequality:
\begin{lem} Let $0<\rho<1$. Assume that the non-negative  function $h(t)\in C[0,T]$ and the positive constants $K_0$ and $K_1$ satisfy
    \[
    h(t)\leq K_0+ K_1\int\limits_0^t(t- s)^{\rho-1}s^{\rho-1} h(s) ds
    \]
    for all $t\in [0, T]$. Then there exists a positive constant $C_{\rho, T}$, depending only on $\rho$, $K_2$ and $T$ such that
    \begin{equation}\label{Gron}
        h(t)\leq K_0 \cdot C_{\rho, T} .
    \end{equation}

\end{lem}

Usually Gronwall's inequality is formulated with a continuous function $k(s)$ instead of $K_1(t-s)^{\rho-1}s^{\rho-1}$. However, estimate (\ref{Gron}) is proved in a similar way to the well-known Gronwall inequality. For the convenience of the reader, we present a proof of estimate (\ref{Gron}).

\begin{proof}
	Iterating the hypothesis of Gronwall's inequality gives
	\[
	 h(t)\leq K_0+K_0 K_1 \int\limits_0^t(t- s)^{\rho-1}s^{\rho-1}  ds + K_1^2\int\limits_0^t(t- s)^{\rho-1}s^{\rho-1}\int\limits_0^s(s-\xi)^{\rho-1} \xi^{\rho-1}h(\xi)d\xi ds
	\]
	\[
	\leq K_{\rho, T}+K_1^2 \int\limits_0^t u(\xi) \xi^{\rho-1} \int\limits_\xi^t (t-s)^{\rho-1}(s-\xi)^{\rho-1} s^{\rho-1} ds d\xi,
	\]
	where
	\[
	K_{\rho, T}=K_0+ K_0 K_1 \int\limits_0^T(t- s)^{\rho-1}s^{\rho-1}  ds.
	\]
	For the inner integral we have (see (\ref{int}))
		\[
	\int\limits_\xi^t (s-\xi)^{\rho-1}(t-s)^{\rho-1}ds =	\int\limits_0^{t-\xi} y^{\rho-1}(t-\xi-y)^{\rho-1}dy\leq \frac{2^{2(1-\rho)}}{\rho} (t-\xi)^{2\rho-1}.
	\]
	Now the hypothesis is
	 \[
	h(t)\leq K_0+ K_1^2 \frac{2^{2(1-\rho)}}{\rho}\int\limits_0^t(t- s)^{2\rho-1} h(s) ds.
	\]
	
By repeating this process so many times that $k\rho>1$, we make sure that there is a positive constant $C_\rho=C(\rho, K_2, T)>0$ such that 
		\[
	h(t)\leq K_0+ C_\rho\int\limits_0^t h(s) ds,
	\]
or	
	 \[
    \frac{h(\xi)}{K_0+C_\rho\int\limits_0^\xi h(s) ds}\leq 1.
    \]
    Multiply this by $C_\rho$ to get
    \[
    \frac{d}{d\xi}\ln \bigg(K_0+C_\rho\int\limits_0^\xi h(s) ds\bigg)\leq C_\rho.
    \]
    Integrate from $\xi=0$ to $\xi=t$ and exponentiate to obtain
    \[
    K_0+C_\rho\int\limits_0^t h(s) ds\leq K_0 e^{C_\rho t}.
    \]
    Finally note that the left side is $\geq h(t)$.
\end{proof}

\section{Auxiliary problem and proof of Theorem \ref{main}}

Let us consider the following auxiliary initial-boundary value  problem

\begin{equation}\label{prob2}
    \left\{
    \begin{aligned}
        &i \partial_t^\rho \omega(x,t) - \omega_{xx}(x,t) = - i \mu(t) q''(x) + f(x,t),\quad (x,t)\in \Omega;\\
        &\omega(0, t)=\omega(\pi,t)=0, \quad 0\leq t\leq T;\\
        &\lim\limits_{t\rightarrow 0}J_t^{\rho-1}\omega(x,t) =\varphi(x), \quad 0\leq x\leq \pi,
    \end{aligned}
    \right.
\end{equation}
where function $\mu(t)$ is the unique solution of the Cauchy problem:
\begin{equation}\label{prob2a}
    \left\{
    \begin{aligned}
        &\partial_t^\rho \mu(t) = p(t),\quad 0 < t \leq T;\\
        &\lim\limits_{t\rightarrow 0}J_t^{\rho-1}\mu(t) =0.
    \end{aligned}
    \right.
\end{equation}

Note that the solution to the Cauchy problem (\ref{prob2a}) has the form (see, e.g., \cite{AshCab}):
\[
\mu(t)=\frac{1}{\Gamma(\rho)}\int\limits_{0}^t  (t-s)^{\rho-1}  p(s)ds,
\]

\begin{defin}\label{def} A functions $\omega(x,t)$ with the properties
	\begin{enumerate}
		\item
		$\partial_t^\rho \omega(x,t), \omega_{xx}(x,t)\in C(\Omega)$,
		\item$t^{1-\rho}\omega_{xx}(x,t)\in C((0, \pi) \times [0,T])$,
		\item
		$t^{1-\rho}\omega(x,t)\in C(\overline{\Omega})$,
	\end{enumerate}
and satisfying conditions
	(\ref{prob2a})  is called \textbf{the
		solution} of  problem (\ref{prob2a}).
\end{defin}

\begin{lem}\label{auxiliary}
	Let $\omega(x,t)$ be a solution of problem  (\ref{prob2}). Then the unique solution $\{u(x,t), p(t)\}$ to the inverse problem (\ref{prob1}), (\ref{ad}) has the form
	\begin{equation}\label{u}
		u(x,t)=\omega(x,t)-i \mu(t) q(x),
	\end{equation}
	\begin{equation}\label{p}
		p(t)=\frac{i}{B[ q(x)]}\{\partial^\rho_t\psi(t)-B[\partial^\rho_t \omega(\cdot,t)]\},
	\end{equation}
	where
	\begin{equation}\label{mu}
		\mu(t)=\frac{i}{B[ q(x)]}[\psi(t)- B[\omega(\cdot,t)]].
	\end{equation}

\end{lem}

\begin{proof} Substitute the function  $u(x,t)$, defined by equality (\ref{u}), into the equation in (\ref{prob1}). Then
	\[
	i \partial_t^\rho \omega(x,t) + \partial_t^\rho \mu(t) q(x) - \omega_{xx}(x,t) + i \mu(t) q''(x)= p(t) q(x)+f(x,t).
	\]
	Since $\partial_t^\rho \mu(t) = p(t)$ (see (\ref{prob2a})), we obtain equation (\ref{prob2}), i.e. function $u(x,t)$, defined by (\ref{u}) is a solution of the equation in (\ref{prob1}).  As for the initial condition, again by virtue of (\ref{prob2a}) we get
	\[
	\lim\limits_{t\rightarrow 0}J_t^{\rho-1} u(x,t)=\lim\limits_{t\rightarrow 0}J_t^{\rho-1}\omega(x,t)-i \lim\limits_{t\rightarrow 0}J_t^{\rho-1}\mu(t) q(x) =\lim\limits_{t\rightarrow 0}J_t^{\rho-1}\omega(x,t)=\varphi(x).
	\]
	On the other hand, conditions $q(0)=q(\pi)=0$ imply  $u(0, t)=u(\pi, t)=0$, $0\leq t\leq T$.
	
	From  Definition \ref{def2} of solution $\omega(x,t)$ and the property of the functions $\mu(t)$ and $q(x)$ it immediately follows that the function $u(x,t)$ satisfies the requirements: 	$\partial_t^\rho u(x,t), u_{xx}(x,t)\in C(\Omega)$, $t^{1-\rho}u(x,t)\in C(\overline{\Omega})$.
	
	Thus function $u(x,t)$, defined as (\ref{u}) is a solution of the initial-boundary value problem (\ref{prob1}).
	
	Let us prove equation (\ref{p}). Rewrite (\ref{u}) as
	\[
	iq(x) \mu(t)=\omega(x,t)- u(x,t).
	\]
	Apply (\ref{ad}) to obtain
	\[
	i \mu(t) B[ q(x)] =B[\omega(\cdot,t)] - \psi(t),
	\]
	or, since $B[ q(x)]\neq 0$, we get (\ref{mu}). Finally, using equality $\partial_t^\rho \mu(t)= p(t)$, we have
	\[
	p(t)=\frac{i}{B[ q(x)]}[\partial_t^\rho \psi(t)- B[\partial_t^\rho \omega(\cdot,t)]],
	\]
	which coincides with (\ref{p}). Moreover, from the definition of solution $\omega(x,t)$ of problem (\ref{prob2}) and the property of function $\psi(t)$ one has $t^{1-\rho} p(t)\in C[0,T]$.
\end{proof}

Thus, to solve the inverse problem (\ref{prob1}), (\ref{ad}), it is sufficient to solve the initial-boundary value problem (\ref{prob2}).

\begin{thm}\label{AP}
    Under the assumptions of Theorem \ref{main}, problem (\ref{prob2}) has a unique solution.
\end{thm}

\begin{proof} Let
\begin{equation}\label{Gsep}
    G(x,s)=\frac{i}{B[ q(x)]} (B[\omega(\cdot,s)]-\psi(s))q''(x) +f(x,s)
\end{equation}
and suppose that $s^{1-\rho} G(x,s)\in C_x^a(\overline{\Omega})$. Then by Lemma \ref{prob.w} problem (\ref{prob2}) is equivalent to the integral equation
    \[
    \omega(x,t)=t^{\rho-1} E_\rho (i t^\rho A)\varphi(x)- i \int\limits_0^t (t-s)^{\rho-1} E_{\rho,\rho} (i(t-s)^\rho A) G(x,s) ds.
    \]
        Rewrite this equation as
    \begin{equation}\label{VE2}
        \omega(x,t)=F(x,t)+\int\limits_0^t (t-s)^{\rho-1} E_{\rho, \rho} (i (t-s)^\rho A) \frac{ q''(x)}{B[ q(x)]} B[\omega(\cdot,s)]ds,
    \end{equation}
    where
    \[
    F(x,t)=t^{\rho-1} E_\rho (i t^\rho A)\varphi(x)- i \int\limits_0^t (t-s)^{\rho-1} E_{\rho,\rho} (i(t-s)^\rho A) \bigg[-\frac{i q''(x)}{B[ q(x)]} \psi(s)+f(x,s)\bigg] ds.
    \]
    In order to apply Lemma \ref{VElem} to equation (\ref{VE2}), we show that $t^{1-\rho}F(x,t)\in C(\overline{\Omega})$. Indeed, by estimate (\ref{E}) one has $ E_\rho (i t^\rho A)\varphi(x) \in C(\overline{\Omega}) $.
According to the conditions of the Theorem \ref{main}  $h(x,s)=s^{1-\rho}[-\frac{i q''(x)}{B[ q(x)]} \psi(s)+f(x,s)]\in C_x^a(\overline{\Omega})$. Therefore, by virtue of estimate (\ref{EAint}), the second term of function $t^{1-\rho}F(x,t)$ also belongs to the class $C(\overline{\Omega})$.
Hence, by virtue of Lemma \ref{VElem}, the Volterra equation (\ref{VE2}) has a unique solution $t^{1-\rho}\omega(x,t)\in C(\overline{\Omega})$.

Let us show that $\partial_t^\rho \omega(x,t),\, \omega_{xx}(t)\in C(\Omega)$. First we consider $F_{xx}(x,t)$ and note, that by estimate (\ref{AE}) we have $\frac{\partial^2}{\partial x^2}  E_\rho (i t^\rho A)\varphi (x)\in C([0, \pi]\times(0, T])$. Since function $h$ defined above, belongs to the class $C_x^a(\overline{\Omega})$, then by Lemma \ref{ep}, the second term of  function $F_{xx}(x,t)$ belongs to $C(\Omega)$ and satisfies the estimate:
\[
\bigg|\bigg|t^{1-\rho}\int\limits_0^t (t-s)^{\rho-1} \frac{\partial^2}{\partial x^2}  E_{\rho,\rho} (i(t-s)^\rho A) \bigg[-\frac{i q''(x)}{B[ q(x)]} \psi(s)+f(x,s)\bigg] ds \bigg|\bigg|_{C(\overline{\Omega})}
\]

    \begin{equation}\label{st1}
    \leq C \bigg[\bigg|\bigg|t^{1-\rho} \frac{q''(x)}{B[ q(x)]}\psi(t)\bigg|\bigg|_{C^a_x(\overline{\Omega})}+ ||t^{1-\rho}f(x,t)||_{C^a_x(\overline{\Omega})} \bigg]\leq C_{a, q, B} \big[||t^{1-\rho}\psi||_{C[0, T]}+||t^{1-\rho}f(x,t)||_{C^a_x(\overline{\Omega})}\big].
\end{equation}

We pass to the second term on the right-hand side of equality (\ref{VE2}). Since $t^{1-\rho}\omega(x,t)\in C(\overline{\Omega})$, the conditions of Theorem \ref{main} imply that $s^{1-\rho}\frac{q''(x)}{B[q(x)]}B[\omega(\cdot,s)]\in C_x^a(\overline{\Omega})$. Then again by Lemma \ref{ep}, this term belongs to $C(\overline{\Omega})$ and satisfies the estimate:
\[
\bigg|\bigg|t^{1-\rho}\int\limits_0^t (t-s)^{\rho-1} \frac{\partial^2}{\partial x^2} E_{\rho,\rho} (i(t-s)^\rho A) \frac{ q''(x)}{B[ q(x)]} B[\omega(\cdot,s)] ds\bigg|\bigg|_{C(H)}
\]
\begin{equation}\label{st2}
    \leq C\bigg|\bigg| \frac{q''(x)}{B[ q(x)]}B[t^{1-\rho}\omega(\cdot,t)]\bigg|\bigg|_{C(\overline{\Omega})}\leq C_{a, q, B} ||t^{1-\rho}\omega(x,t)||_{C(\overline{\Omega})}.
\end{equation}

Thus, $\omega_{xx}(x,t)\in C((0, T]; H)$. On the other hand, by virtue of equation (\ref{prob2}) and the conditions of Theorem \ref{main}, we will have
\[
\partial_t^\rho \omega(x,t) = \omega_{xx}(x,t) - i \mu(t)  q''(x) + f(x,t)\in C(\overline{\Omega})).
\]
The fact that here $\mu\in C[0,T]$ follows again from the conditions of the Theorem \ref{main} and equality (\ref{mu}).

It remains to show that $t^{1-\rho} G(x,t)\in C_x^a(\overline{\Omega})$. But this fact follows from the conditions of Theorem \ref{main} and the already proven assertion: $t^{1-\rho}\omega(x,t)\in C(\overline{\Omega})$.

\end{proof}

As noted above Theorem \ref{main} is an immediate consequence of Lemma \ref{auxiliary} and Theorem \ref{AP}.

\section{Proof of Theorem \ref{estimate}}

First we prove the following statement on the stability of the solution to problem (\ref{prob2}), (\ref{prob2a}).

\begin{thm}\label{estimate2} Let assumptions of Theorem \ref{estimate} be satisfied. Then  the solution to problem (\ref{prob2}), (\ref{prob2a}) obeys the stability estimate
    \begin{equation}\label{se}
        ||t^{1-\rho}\partial_t^\rho \omega||_{C(\overline{\Omega})}\leq C_{\rho, q, B} \big[   ||\varphi_{xx}||_{C^a[0,\pi]}  + ||t^{1-\rho}\psi||_{C[0,T]} + ||t^{1-\rho}f(x,t)||_{C^a_x(\overline{\Omega})}\big],
    \end{equation}
    where $C_{\rho, q, B, \epsilon}$ is a constant, depending only on $\rho, q$ and $B$.
\end{thm}

\begin{proof}Let us begin the proof of the inequality (\ref{se}) by establishing an estimate for $\omega_{xx}(x,t)$ and then use it with equation (\ref{prob2}). To this end we have from (\ref{AEA})
    \[
    \bigg|\bigg|\frac{\partial^2}{\partial x^2} E_\rho (i t^\rho A)\varphi\bigg|\bigg|_{C(\overline{\Omega})}\leq C ||\varphi_{xx}||_{C^a[0,\pi]}.
    \]
This estimate together with (\ref{st1}) implies
\[
||t^{1-\rho}F_{xx}(x,t)||_{C(\overline{\Omega})}\leq C ||\varphi_{xx}||_{C^a[0,\pi]}+ C_{a, q, B}\big[ ||t^{1-\rho}\psi||_{C[0, T]}+||t^{1-\rho}f(x,t)||_{C^a_x(\overline{\Omega})}\big].
\]
Then, using equality (\ref{VE2}) and inequality (\ref{st2}), we get
\begin{equation}\label{Aomega}
    ||t^{1-\rho}\omega_{xx}(x,t)||_{C(\overline{\Omega})}\leq C ||\varphi_{xx}||_{C^a[0,\pi]}+ C_{a, q, B}\big[ ||t^{1-\rho}\psi||_{C[0, T]}+||t^{1-\rho}f(x,t)||_{C^a_x(\overline{\Omega})} + ||t^{1-\rho}\omega(x,t)||_{C(\overline{\Omega})}\big].
    \end{equation}
As a result, we have obtained an estimate for $\omega_{xx}(x,t)$ through $\omega(x,t)$.
To estimate $||t^{1-\rho}\omega(x,t)||_{C(\overline{\Omega})}$, we will proceed as follows. Apply estimates (\ref{E}) and (\ref{EAint}) to get
\[
||t^{1-\rho} F(x,t)||_{C(\overline{\Omega})}\leq ||\varphi||_{C^a[0,\pi]}+\frac{T^\rho}{\rho}\big[ C_{q, B}||q''||_{C^a[0,\pi]}||t^{1-\rho} \psi||_{C[
    0,T]} +||t^{1-\rho} f||_{C^a_x(\overline{\Omega})}\big]
\]
Again by estimate (\ref{E}) we have
\[
\bigg|\bigg|t^{1-\rho}\int\limits_0^t (t-s)^{\rho-1}  E_{\rho,\rho} (i(t-s)^\rho A)  \frac{ q''(x)}{B[ q(x)]} B[\omega(\cdot,s)] ds\bigg|\bigg|_{C[0,\pi]}\leq
\]
\[ C_{q, B} ||q''||_{C^a[0,\pi]}\int\limits_0^t (t-s)^{\rho-1}||\omega(x,s)||_{C[0,\pi]}.
\]
Therefore, from equation (\ref{VE2}) we obtain an estimate
\[
    ||t^{1-\rho}\omega(x,t)||_{C[0,\pi]}\leq ||\varphi||_{C^a[0,\pi]}+C_{q,\rho, B} [||t^{1-\rho}\psi||_{C[0,T]}+||t^{1-\rho}f||_{C^a_x(\overline{\Omega})}]+
    \]
    \[
    C_{q,B}\int\limits_0^t (t-s)^{\rho-1}s^{\rho-1}||s^{1-\rho}\omega(x,s)||_{C[0,\pi]}ds
\]
for all $t\in [0, T]$. Finally, the Gronwall inequality (\ref{Gron}) implies
\[
    ||t^{1-\rho}\omega(x,t)||_{C(\overline{\Omega})}\leq C_{q, \rho, B}\big[||\varphi||_{C^a[0,\pi]}+||t^{1-\rho}\psi||_{C[0,T]}+||t^{1-\rho}f||_{C^a_x(\overline{\Omega})}\big].
\]

We substitute this estimate in (\ref{Aomega}) and apply $||\varphi||_{C^a[0,\pi]}\leq C ||\varphi_{xx}||_{C^a[0,\pi]}$ to get
\[
    ||t^{1-\rho} \omega_{xx}||_{C(\overline{\Omega})}\leq C_{\rho, q, B} \big[  ||\varphi_{xx}||_{C^a[0,\pi]}  + ||t^{1-\rho}\psi||_{C[0,T]} + ||t^{1-\rho}f||_{C^a_x(\overline{\Omega})}\big],
\]
To obtain  estimate (\ref{se}), it remains to note that
\[
\partial_t^\rho \omega(x,t) = \omega_{xx}(x,t) - i \mu(t)  q''(x) + f(x,t)
\]
and use the estimate
\[
||t^{1-\rho}\mu||_{C[0, T]}\leq C_{q,B}\big[  ||t^{1-\rho}\psi||_{C[0, T]} + ||t^{1-\rho} \omega||_{C(\overline{\Omega})}\big],
\]
which follows from definition (\ref{mu}) and the conditions of Theorem \ref{main}.
\end{proof}

\textbf{Proof of Theorem \ref{estimate}.}

Apply (\ref{p}) to get
\[
||t^{1-\rho}p(t)||_{C[0, T]}\leq C_{q, B}\big[ ||t^{1-\rho}\partial_t^\rho \omega||_{C(\overline{\Omega})}+||t^{1-\rho}\partial_t^\rho \psi||_{C[0, T]}\big].
\]
Equations (\ref{u}) and (\ref{prob2a}) imply
\[
\partial^\rho_t u(x,t)=\partial^\rho_t \omega(x,t)+p(t) q(x).
\]
Hence, from estimates of $\partial^\rho_t \omega(x,t)$ and $p(t)$, we obtain an estimate for $\partial^\rho_t u(x,t)$. On the other hand, by virtue of equation (\ref{prob1}), we will have
\[
-u_{xx}(x,t)=-i \partial_t^\rho u(x,t) + p(t) q(x) +f(x,t).
\]

Now, to establish estimate (\ref{se}), it suffices to use the statement of Theorem \ref{estimate2}.

\section{Acknowledgement}
The authors are grateful to A. O. Ashyralyev for posing the
problem and they convey thanks to Sh. A. Alimov for discussions of
these results.
The authors acknowledge financial support from the  Ministry of Innovative Development of the Republic of Uzbekistan, Grant No F-FA-2021-424.

\end{document}